\documentclass[11pt]{article}
\usepackage{amsthm,geometry,amssymb,amsmath,enumerate,float,tikz,cite,setspace}
\tikzstyle{every node}=[circle, draw, inner sep=0pt, minimum width=4pt, fill=black, thick]
\newtheorem{theorem}{Theorem}
\newtheorem{claim}{Claim}

\title{Acyclic matchings in graphs of bounded maximum degree\thanks{Funded by the Deutsche Forschungsgemeinschaft (DFG, German Research Foundation) - 388217545.}}
\author{Julien Baste \and Maximilian F\"{u}rst  \and Dieter Rautenbach}
\date{}

\begin{document}
\onehalfspace
\maketitle
\begin{center}
{\small 
Institute of Optimization and Operations Research, Ulm University, Germany\\
\texttt{julien.baste,maximilian.fuerst,dieter.rautenbach@uni-ulm.de}\\[3mm]
}
\end{center}

\begin{abstract}
A matching $M$ in a graph $G$ is acyclic if the subgraph of $G$ 
induced by the set of vertices that are incident to an edge in $M$ 
is a forest.
We prove that every graph with $n$ vertices, maximum degree at most $\Delta$, and no isolated vertex,
has an acyclic matching of size at least 
$(1-o(1))\frac{6n}{\Delta^2},$
and we explain how to find such an acyclic matching in polynomial time. 
\end{abstract}
{\small 
\begin{tabular}{lp{13cm}}
{\bf Keywords:} Acyclic Matching
\end{tabular}
} \\
{\small 
\begin{tabular}{lp{13cm}}
{\bf AMS subject classification:} 05C70
\end{tabular}
\vfill$\mbox{}$
}

\pagebreak

\section{Introduction}
We consider simple, finite, and undirected graphs, and use standard terminology.
Let $M$ be a matching in a graph $G$, 
and let $H$ be the subgraph of $G$ 
induced by the set of vertices that are incident to an edge in $M$.
If $H$ is a forest, then $M$ is an \textit{acyclic} matching in $G$ \cite{gohehela},
and, 
if $H$ is $1$-regular, then $M$ is an \textit{induced} matching in $G$ \cite{stva}.
If $\nu(G)$, $\nu_{ac}(G)$, and $\nu_{s}(G)$
denote the largest size of a matching,
an acyclic matching,
and an induced matching in $G$,
respectively, 
then, since every induced matching is acyclic,
we have 
$$\nu(G)\geq \nu_{ac}(G) \geq \nu_s(G).$$
In contrast to the matching number $\nu(G)$, 
which is a well known classical tractable graph parameter,
both, 
the acyclic matching number $\nu_{ac}(G)$
as well as 
the induced matching number $\nu_{s}(G)$
are computationally hard \cite{gohehela,stva,dadelo,papr}.
While induced matchings have been studied in great detail,
see, in particular, \cite{jo,jo2,jorasa,hr} for lower bounds on $\nu_{s}(G)$
for graphs $G$ of bounded maximum degree as well as the references therein,
only few results are known on the acyclic matching number.
While the equality $\nu(G)=\nu_s(G)$ can be decided efficiently for a given graph $G$ \cite{cawa,koro},
it is NP-complete to decide whether $\nu(G)=\nu_{ac}(G)$
for a given bipartite graph $G$ of maximum degree at most $4$ \cite{fura2},
and efficient algorithms computing the acyclic matching number are known only
for certain graph classes \cite{papr,fu,fura2,bara}.
%chain graphs and bipartite permutation graphs \cite{papr},
%$P_5$-free graphs and $2P_3$-free graphs \cite{fu,fura2}, 
%and chordal graphs \cite{bara}.
It is known \cite{bara} that
$\nu_{ac}(G) \geq \frac{m}{\Delta^2}$ for a graph $G$ with $m$ edges and maximum degree $\Delta$,
which was improved \cite{fura} to $\frac{m}{6}$ for connected subcubic graphs $G$ of order at least $7$.
Since, for every $\Delta$-regular graph $G$ with $m$ edges, 
a simple edge counting argument implies $\nu_{ac}(G)\leq \frac{m-1}{2(\Delta-1)}$,
the constructive proofs of these bounds yield an efficient
$\frac{\Delta^2}{2(\Delta-1)}$-factor approximation algorithm 
for $\Delta$-regular graphs,
and 
an efficient $\frac{3}{2}$-factor approximation algorithm for cubic graphs
for the maximum acyclic matching problem.

In the present paper we show a lower bound on the acyclic matching number of a graph $G$ 
with $n$ vertices, maximum degree $\Delta$, and no isolated vertex,
which is inspired by a result of Joos \cite{jo} who proved
\begin{align} \label{felixBound}
\nu_s(G) \geq \frac{n}{\left(\lfloor \frac{\Delta}{2} \rfloor + 1\right)\left(\lceil \frac{\Delta}{2} \rceil + 1\right)}
\end{align}
provided that $\Delta\geq 1000$.
(\ref{felixBound}) is tight for the graph that arises 
by attaching $\lfloor \frac{\Delta}{2} \rfloor$ new vertices of degree $1$
to every vertex of a complete graph of order $\lceil \frac{\Delta}{2} \rceil+1$.
In view of these graphs, 
we conjectured \cite{fu,fura} that twice the right hand side of (\ref{felixBound}) 
should be the right lower bound on the acyclic matching number of the considered graphs
for sufficiently large $\Delta$,
that is, 
we believe that our following main result 
can be improved by a factor of roughly $\frac{4}{3}$.

\begin{theorem} \label{t1}
If $G$ is a graph with $n$ vertices, maximum degree at most $\Delta$, and no isolated vertex, then
$$\nu_{ac}(G) \geq \frac{6n}{\Delta^2+12\Delta^{\frac{3}{2}}}.$$
\end{theorem}
Note that, for graphs that are close to $\Delta$-regular, 
the bound $\nu_{ac}(G) \geq \frac{m}{\Delta^2}$ is stronger
than Theorem \ref{t1}.
We prove Theorem \ref{t1} in the next section.
In the conclusion we discuss algorithmic aspects of its proof 
and possible generalizations to so-called degenerate matchings \cite{bara}.

\section{Proof of Theorem \ref{t1}}

We prove the theorem by contradiction.
Therefore, suppose that $G$ is a counterexample of minimum order.
Clearly, $G$ is connected.
If $\Delta=1$, then $G$ is $K_2$, and, hence, $\nu_{ac}(G)=\frac{n}{2}$.
If $\Delta=2$, then $G$ is a path or a cycle, 
which implies $\nu_{ac}(G)\geq \frac{n-2}{2}$.
These observations imply $\Delta\geq 3$.
At several points within the proof we consider 
an acyclic matching $M$ in $G$, 
and we consistently use 
\begin{itemize}
\item $V_M$ to denote the set of vertices of $G$ that are incident to an edge in $M$,
\item $N_M$ to denote the set of vertices in $V(G)\setminus V_M$ that have a neighbor in $V_M$,
\item $G_M$ to denote the graph $G-(V_M\cup N_M)$, 
\item $I_M$ to denote the set of isolated vertices of $G_M$, and
\item $G'_M$ to denote the graph $G_M-I_M$.
\end{itemize}
Since $G'_M$ is no counterexample, 
and the union of $M$ with any acyclic matching in $G'_M$
is an acyclic matching in $G$, we obtain
$$\frac{6n}{\Delta^2+12\Delta^{\frac{3}{2}}}>\nu_{ac}(G)\geq |M|+\frac{6(n-|V_M\cup N_M\cup I_M|)}{\Delta^2+12\Delta^{\frac{3}{2}}},$$
which implies
\begin{align}\label{imporEQ}
|V_M|+|N_M|+|I_M|>\left( \frac{\Delta^2}{6} + 2\Delta^{\frac{3}{2}} \right)|M|.
\end{align}

\begin{claim} \label{claim1}
For every edge $uv$ in $G$, we have $d_G(u) + d_G(v)>2\sqrt{\Delta}$.
\end{claim}
\begin{proof}
Suppose, for a contradiction, that $d_G(u) + d_G(v)\leq 2\sqrt{\Delta}$
for some edge $uv$ of $G$.
For $M=\{ uv\}$, we obtain
$|V_M|+|N_M|+|I_M|\leq 
2+\left(2\sqrt{\Delta}-2\right)+\left(2\sqrt{\Delta}-2\right)(\Delta-1)\leq 2 \Delta^{\frac{3}{2}}$,
contradicting (\ref{imporEQ}).
\end{proof}
Let $S$ be the set of vertices of degree at most $\sqrt{\Delta}$.
By Claim \ref{claim1}, the set $S$ is independent.

\begin{claim} \label{claim3}
$S$ is not empty.
\end{claim}
\begin{proof}
Suppose, for a contradiction, 
that the minimum degree $\delta$ of $G$ is larger than $\sqrt{\Delta}$.
Let $uv$ be an edge of $G$ such that $u$ is of minimum degree.
Let $M=\{ uv\}$.
Since every vertex in $I_M$ has degree at least $\delta$, we have 
$$|V_M|+|N_M|+|I_M| 
\leq 2+(\Delta+\delta-2) + \frac{(\Delta+\delta-2)(\Delta-1)}{\delta}
\leq \frac{(\Delta+\delta)^2}{\delta}.$$
If $\Delta=3$, then $\delta$ is $2$ or $3$, and in both cases
$2+(\Delta+\delta-2) + \frac{(\Delta+\delta-2)(\Delta-1)}{\delta}$
is less than the right hand side of (\ref{imporEQ}),
contradicting (\ref{imporEQ}).
For $\Delta\geq 4$,
we obtain that 
$\frac{(\Delta+\delta)^2}{\delta}
\leq 
\frac{\left(\Delta+\sqrt{\Delta}\right)^2}{\sqrt{\Delta}}$
is less than the right hand side of (\ref{imporEQ}).
Hence, also in this case, we obtain a contradiction (\ref{imporEQ}).
\end{proof}
Let $N$ be the set of vertices that have a neighbor in $S$,
and, for a vertex $v$ in $G$, 
let $d_S(v)$ be the number of neighbors of $v$ in $S$.
Since $S$ is independent, the sets $S$ and $N$ are disjoint.
\begin{claim} \label{claim2}
$\max\{ d_S(v):v\in V(G)\}=\alpha \Delta$ for some $\alpha$ with $0.2\leq \alpha\leq 0.8$.

In other words, we have 
$d_S(v)\leq 0.8\Delta$ for every vertex $v$ of $G$, and
$d_S(v)\geq 0.2\Delta$ for some vertex $v$ of $G$.
\end{claim}
\begin{proof}
Let the vertex $v$ maximize $d_S(v)$.
Suppose, for a contradiction, 
that $d_S(v) =\alpha\Delta$
for some $\alpha$ with either $\alpha<0.2$ or $\alpha>0.8$.
Let $u$ be a neighbor of $v$ of minimum degree. 
By Claim \ref{claim3}, we have $d_S(v)\geq 1$,
which implies $d_G(u)\leq \sqrt{\Delta}$.
Let $M=\{ uv\}$.
Clearly, 
$$|V_M|+|N_M|\leq \sqrt{\Delta}+\Delta.$$
Let 
$I_1$ be the set of vertices in $I_M$ that have a neighbor in $N_G(u)\cup (N_G(v)\cap S)$,
let $I_2=(I_M\setminus I_1)\cap S$,
and   
let $I_3=I_M\setminus (I_1\cup I_2)$.

We obtain 
\begin{eqnarray*}
|I_1| &\leq &(\Delta-1)(d_G(u)-1)+\left(\sqrt{\Delta}-1\right)|N_G(v)\cap S|\\
&\leq &(\Delta-1)\left(\sqrt{\Delta}-1\right)+\left(\sqrt{\Delta}-1\right)\alpha\Delta\\
&\leq &(1+\alpha)\Delta^{\frac{3}{2}}-\left(\sqrt{\Delta}+\Delta\right).
\end{eqnarray*}
Let $N'=N_G(v)\setminus (N_G(u)\cup S)$.
Note that $|N'|\leq (1-\alpha)\Delta$,
and that the vertices in $I_2\cup I_3$ have all their neighbors in $N'$.
By the choice of $v$,
every vertex in $N'$
has at most $\alpha\Delta$ neighbors in $S$,
which implies 
$$|I_2|\leq \alpha\Delta |N'|
\leq \alpha(1-\alpha)\Delta^2.$$
Since there are at most 
$\Delta|N'|$
edges between 
$N'$ and $I_3$,
and every vertex in $I_3$ has degree more than $\sqrt{\Delta}$,
we obtain 
$$|I_3|<\frac{\Delta|N'|}{\sqrt{\Delta}}\leq (1-\alpha)\Delta^{\frac{3}{2}}.$$
Altogether, we obtain
\begin{eqnarray*}
|V_M|+|N_M|+|I_M| & \leq &
\sqrt{\Delta}+\Delta+
(1+\alpha)\Delta^{\frac{3}{2}}-\left(\sqrt{\Delta}+\Delta\right)
+\alpha(1-\alpha)\Delta^2
+(1-\alpha)\Delta^{\frac{3}{2}}\\
&=& \alpha(1-\alpha)\Delta^2+2\Delta^{\frac{3}{2}}\\
&\leq & 0.16\Delta^2+2\Delta^{\frac{3}{2}},
\end{eqnarray*}
contradicting (\ref{imporEQ}).
\end{proof}
Note that, so far in the proof of each claim, we had $|M|=1$, 
and iteratively applying the corresponding reductions 
would eventually lead to an induced matching in $G$ similarly as in \cite{jo}.
In order to improve (\ref{felixBound}),
we now choose $M$ non-locally in some sense:
Let $M$ be an acyclic matching in $G$ such that 
\begin{enumerate}
 \item[(i)] $M$ only contains edges incident to a vertex in $S$, 
 \item[(ii)] every vertex in $V_M\cap S$ has degree one in the subgraph of $G$ induced by $V_M$,
 \item[(iii)] every vertex $v$ in $V_M\cap N$ satisfies $d_S(v) \geq 0.2 \Delta$, and
\end{enumerate}
$M$ maximizes 
\begin{align} \label{matchingMaximizer}
 \sum_{v \in V_M \cap N}{d_S(v)}.
\end{align}
among all acyclic matchings satisfying (i), (ii), and (iii).
By Claim \ref{claim2}, the matching $M$ is non-empty.

We now define certain relevant sets, see Figure \ref{mainFig} for an illustration.
\begin{itemize}
\item Let $X$ be the set of vertices in $N_M$ 
that are not adjacent to a vertex in $V_M\cap S$
and that have at least one neighbor in $S$ that is not adjacent to a vertex in $V_M$.

({\it Note that $X\subseteq N$, and that the edges between vertices in $X$ and suitable neighbors in $S$
are possible candidates for modifying $M$.})
\item Let $Y$ be the set of vertices in $N_M\setminus X$ 
that are not adjacent to a vertex in $V_M \cap S$.

({\it Note that $Y$ contains $N_M\setminus N=(N_M\cap S)\cup (N_M\setminus (S\cup N))$.})
\item Let $Z=(N\cap N_M)\setminus (X\cup Y)$.

({\it Note that $Z$ consists of the vertices in $N_M$ that have a neighbor in $V_M\cap S$.})
\item Let $I_1$ be the set of vertices in $I_M\cap S$ 
that have a neighbor in $N_M \setminus X$.

({\it Note that, by the definition of $X$, no vertex in $I_1$ can have a neighbor in $Y\cap N$,
which implies that every vertex in $I_1$ has a neighbor in $Z$.})
\item Let $I_2$ be the set of vertices in $I_M\setminus S$ 
that have a neighbor in $Z$.  
\item Let $I_3$ be the set of vertices in $I_M\cap S$ 
that only have neighbors in $X$.

({\it Note that $I_1\cup I_3=I_M\cap S$.})
\item Finally, let $I_4 = I_M\setminus (I_1\cup I_2 \cup I_3)$.
\end{itemize}
\begin{figure}[H]
 \centering
 \begin{tikzpicture}[scale=0.8]
 \draw[draw=black, thick] (0,0) rectangle ++(9,9);
 \draw[draw=black, dashed, thick] (3,0) -- (3,9);
 \draw[draw=black, dashed, thick] (6,0) -- (6,9);
 \draw[draw=black] (0,3) -- (9,3);
 \draw[draw=black] (0,6) -- (9,6);

 \foreach \i in {0,1,2} {
    \node (u\i) at (1.5,8.5-\i/3) {};
    \node (v\i) at (4.5,8.5-\i/3) {};
    \draw[-,thick] (u\i) -- (v\i);
 }
\draw[-,dotted, thick] (3.75,7.58) -- (3.75,7.08);
\draw[-,dotted, thick] (2.25,7.58) -- (2.25,7.08);

 \foreach \i in {0,1} {
    \node (x\i) at (1.5,6.833 -\i/3) {};
    \node (y\i) at (4.5,6.833 -\i/3) {};
    \draw[-,thick] (x\i) -- (y\i);
 }

 \draw[thick, shading=radial,outer color=gray!50 ,inner color=white] (1.5,2.1) circle(0.5);
 \draw[thick, shading=radial,outer color=gray!50 ,inner color=white] (1.5,0.9) circle(0.5);

\pgftext[x=1.5cm,y=2.1cm] {$I_1$};
\pgftext[x=1.5cm,y=0.9cm] {$I_3$};

 \draw[thick, shading=radial,outer color=gray!50 ,inner color=white] (4.5,5) circle(0.8);

 \pgftext[x=4.5cm,y=5cm] {$X$};

 \draw[thick, shading=radial,outer color=gray!50 ,inner color=white] (6,0.9) ellipse(2 and 0.35);
 \draw[thick, shading=radial,outer color=gray!50 ,inner color=white] (6,2.1) ellipse(2 and 0.35);

 \draw[thick, shading=radial,outer color=gray!50 ,inner color=white, rounded corners=10pt]
       (4.5,3.3) -- (0.5,3.3) -- (0.5,5.5) -- (2.5,5.5) -- (2.5,3.9) -- (6.5,3.9) -- (6.5,5.5) -- (8.5,5.5) -- (8.5,3.3) -- (4.5,3.3);

 \pgftext[x=6cm,y=2.1cm] {$I_2$};
 \pgftext[x=6cm,y=0.9cm] {$I_4$};

 \pgftext[x=4.5cm,y=3.6cm] {$Y$};
 \pgftext[x=0.75cm,y=8cm] {$M$};
% \pgftext[x=-0.75cm,y=7.5cm] {$V_M$};
 \pgftext[x=-0.75cm,y=4.5cm] {$N_M$};
% \pgftext[x=-0.75cm,y=1.5cm] {$G_M$};
 \pgftext[x=1.5cm,y=9.5cm] {$S$};
 \pgftext[x=4.5cm,y=9.5cm] {$N$};
% \pgftext[x=7.5cm,y=9.5cm] {$G-(S \cup N)$};

\end{tikzpicture}
 \caption{An illustration of the different relevant sets.} \label{mainFig}
\end{figure}
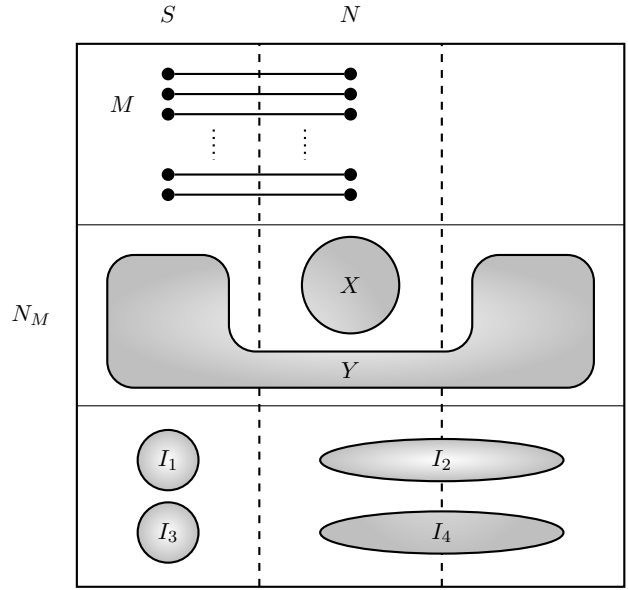 
Clearly,
\begin{eqnarray}\label{e1}
|V_M|+|N_M| & \leq & \left(\sqrt{\Delta}+\Delta\right)|M|.
\end{eqnarray}
Since every vertex in $I_1\cup I_2$ has a neighbor in $Z$,
and every vertex in $Z$ has a neighbor in $V_M\cap S$, we have
\begin{eqnarray}\label{e2}
|I_1\cup I_2| \leq (\Delta-1)|Z|\leq (\Delta-1)\left(\sqrt{\Delta} -1\right)|M|
=\left(\Delta^{\frac{3}{2}}-\Delta-\sqrt{\Delta}+1\right)|M|.
\end{eqnarray}
Since every vertex in $I_4$ has degree more than $\sqrt{\Delta}$
and has all its neighbors in $X \cup Y$,
and every vertex in $X \cup Y$ has a neighbor in $V_M\cap N$, we have
\begin{eqnarray}\label{e3}
|I_4| \leq \frac{(\Delta-1)|X\cup Y|}{\sqrt{\Delta}}\leq \frac{(\Delta-1)^2|M|}{\sqrt{\Delta}}
=\left(\Delta^{\frac{3}{2}}-2\sqrt{\Delta}+\frac{1}{\sqrt{\Delta}}\right)|M|.
\end{eqnarray}
Combining (\ref{e1}), (\ref{e2}), and (\ref{e3}), we obtain 
\begin{eqnarray}\label{e4}
|V_M|+|N_M|+|I_M|-|I_3| & \leq & 2\Delta^{\frac{3}{2}}.
\end{eqnarray}
In order to estimate $|I_3|$, we partition the set $X$ as follows:
\begin{itemize}
\item Let $X_1$ be the set of vertices $v$ in $X$ with $d_S(v) < 0.2 \Delta$,
\item let $X_2$ be the set of vertices in $X\setminus X_1$ with at least four neighbors in $V_M$, and 
\item let $X_3 = X\setminus (X_1 \cup X_2)$.
\end{itemize}
For a vertex $v$ in $V_M \cap N$, let $d_3(v)$ be the number of neighbors of $v$ in $X_3$.
\begin{claim} \label{claim4}
$|I_3| \leq 0.2 \Delta |X_1| + 0.8\Delta |X_2| + \frac{2}{3} \sum\limits_{v \in V_M \cap N}{d_S(v) d_3(v)}.$
\end{claim}
\begin{proof}
By Claim \ref{claim2}, we obtain that 
$$
|I_3| \leq \sum_{w \in X}{d_S(w)}
=\sum_{w \in X_1\cup X_2\cup X_3}{d_S(w)}
\leq 0.2 \Delta |X_1| + 0.8 \Delta |X_2| + \sum_{w \in X_3}{d_S(w)}.$$
Let $w$ be a vertex in $X_3$. 
By the definition of $X$, 
the vertex $w$ has a neighbor $u$ in $S$ that is not adjacent to a vertex in $V_M$.
If $w$ has only one neighbor in $V_M$,
then $M \cup \{wu\}$ is an acyclic matching satisfying (i), (ii), and (iii) that has a larger value in (\ref{matchingMaximizer}), 
contradicting the choice of $M$.
Hence, we may assume that $w$ has either $k=2$ or $k=3$ neighbors $v_1,\ldots,v_k$ in $V_M$. 
Let $u_1v_1,\ldots ,u_kv_k$ be edges in $M$, and suppose that $d_S(v_1) \leq \ldots \leq d_S(v_k)$.
Since 
$$M'=(M \cup \{wu\})\setminus \{u_1v_1,\ldots ,u_{k-1}v_{k-1}\}$$
is an acyclic matching satisfying (i), (ii), and (iii),
the choice of $M$ implies that the value of $M'$ in (\ref{matchingMaximizer}) is at most the one of $M$,
which implies 
$$d_S(w)\leq \sum\limits_{i=1}^{k-1}{d_S(v_i)}\leq \frac{k-1}{k}\sum\limits_{i=1}^k{d_S(v_i)}\leq \frac{2}{3}\sum\limits_{i=1}^k{d_S(v_i)}.$$
Now, we obtain 
$$
\sum_{w \in X_3}{d_S(w)} \leq \frac{2}{3} \sum_{w \in X_3}\,\,\,\sum_{v \in V_M \cap N \cap N_G(w)}\,\,\,d_S(v) 
= \frac{2}{3} \sum_{v \in V_M \cap N}{d_3(v)d_S(v)},
$$
which completes the proof.
\end{proof}
For a vertex $v$ in $V_M \cap N$, let $d_1(v)$ be the number of neighbors of $v$ in $X_1 \cup X_2$.
By property (iii), we have $d_S(v) \geq 0.2 \Delta$, which implies that $d_1(v) \leq 0.8 \Delta$. 
Using Claim \ref{claim4},
$xy\leq \frac{(x+y)^2}{4}$ for $x,y\geq 0$, and
$d_S(v)+d_1(v)+d_3(v)\leq \Delta$ and
$d_1(v)^2\leq 0.8 \Delta d_1(v)$ for $v\in V_M\cap N$,
we obtain 
\begin{align*}
|I_3| &\leq 0.2 \Delta |X_1| + 0.8\Delta |X_2| + \frac{2}{3} \sum_{v \in V_M \cap N}{d_S(v) d_3(v)} \\
      &\leq 0.2 \Delta (|X_1| + 4 |X_2|) + \frac{1}{6} \sum_{v \in V_M \cap N}{(d_S(v)+d_3(v))^2} \\
      &\leq 0.2 \Delta \sum_{v \in V_M \cap N}{d_1(v)} + \frac{1}{6} \sum_{v \in V_M \cap N}{(\Delta - d_1(v))^2} \\
      &=\frac{\Delta^2}{6}|M| + \Delta \left(\frac{1}{5}-\frac{1}{3}\right) \sum_{v \in V_M \cap N}{d_1(v)} + \frac{1}{6} \sum_{v \in V_M \cap N}{d_1(v)^2} \\
      &\leq \frac{\Delta^2}{6}|M| + \Delta \left( \frac{2}{15} -\frac{2}{15} \right) \sum_{v \in V_M \cap N}{d_1(v)} \\
      &=\frac{\Delta^2}{6}|M|,
\end{align*}
and together with (\ref{e4}), we obtain a final contradiction to (\ref{imporEQ}) completing the proof. $\Box$

\section{Conclusion}

While the choice of $M$ after Claim \ref{claim2} in the proof is non-constructive,
the proof of Theorem \ref{t1} easily yields an efficient algorithm that
returns an acyclic matching in a given input graph $G$ as considered in Theorem \ref{t1}
with size at least $\frac{6n}{\Delta^2+12\Delta^{\frac{3}{2}}}$.
If the statements of Claims \ref{claim1}, \ref{claim3}, or \ref{claim2} fail,
then their proofs contain simple reduction rules, 
each fixing one edge in the final acyclic matching and producing a strictly smaller instance $G'_M$.
Adding that fixed edge to the output on the instance $G'_M$ yields the desired acyclic matching.
The matching $M$ chosen after Claim \ref{claim2} can be initialized as any acyclic matching satisfying (i), (ii), and (iii).
If Claim \ref{claim4} fails, then its proof contains simple update procedures
that increase the value in (\ref{matchingMaximizer}).
Since this value is integral and polynomially bounded,
after polynomially many updates the statement of Claim \ref{claim4} holds,
and adding $M$ to the output on the instance $G'_M$ yields the desired acyclic matching.

The acyclic matchings $M$ produced by the proof of Theorem \ref{t1} actually have a special structure 
because the subgraph $H$ of $G$ induced by the set of vertices that are incident to an edge in $M$ 
is not just any forest but a so-called {\it corona} of a forest,
that is, every vertex $v$ of $H$ of degree at least $2$ in $H$ has a unique neighbor $u$ of degree $1$ in $H$,
and all the edges $uv$ form $M$.

As a generalization of acyclic matchings, 
\cite{bara} introduced the notion of a {\it $k$-degenerate matching}
as a matching $M$ in a graph $G$ such that 
the subgraph $H$ of $G$ defined as above is $k$-degenerate.
If the {\it $k$-degenerate matching number $\nu_k(G)$} of $G$ 
denotes the largest size of a $k$-degenerate matching in $G$,
then $\nu_1(G)$ coincides with the acyclic matching number.
We conjecture that 
$$\nu_k(G) \geq \frac{(k+1)n}{\left(\lfloor \frac{\Delta}{2} \rfloor + 1\right)\left(\lceil \frac{\Delta}{2} \rceil + 1\right)}$$
for every graph $G$ with $n$ vertices, sufficiently large maximum degree $\Delta$, 
and no isolated vertex.
A straightforward adaptation of the proof of Theorem \ref{t1} yields
$$
\frac{\nu_k(G)}{n}\geq 
\begin{cases}
(1-o(1))\dfrac{4(k+3)}{3\Delta^2} & \mbox{for }k\in \{ 2,3,4,5,6\}\mbox{ and}\\[2mm]
(1-o(1))\dfrac{k+4}{\Delta^2} & \mbox{for }k\geq 7.
\end{cases}
$$
for these graphs $G$.

\end{document}